\documentclass{amsart}
\usepackage[english]{babel}
\usepackage{amsmath, amsthm}
\usepackage{enumerate}
\usepackage{palatino}
\usepackage{mathpazo}
\usepackage{paralist}
\usepackage[a4paper]{geometry}
\usepackage{url}
\usepackage{hyperref}

\DeclareMathOperator{\id}{id}
\DeclareMathOperator{\im}{im}

\DeclareMathOperator{\cl}{cl}
\DeclareMathOperator{\zcl}{zcl}

\newcommand{\K}{\mathbb{K}}
\newcommand{\F}{\mathbb{F}}
\newcommand{\NN}{\mathbb{N}}
\newcommand{\ZZ}{\mathbb{Z}}
\newcommand{\RR}{\mathbb{R}}
\newcommand{\CP}{\mathbb{C}P}
\newcommand{\RP}{\mathbb{R}P}

\newcommand{\Isom}{\mathrm{Isom}}

\newcommand{\Sc}{\mathcal{S}}

\newcommand{\Spin}{\mathrm{Spin}}
\newcommand{\Cl}{\mathrm{Cl}}

\newcommand{\quot}[2]{\left.\raisebox{.2em}{$#1$}\middle/\raisebox{-.2em}{$#2$}\right.}

\DeclareMathOperator{\TC}{\mathsf{TC}}
\DeclareMathOperator{\cat}{\mathsf{cat}}
\DeclareMathOperator{\secat}{\mathsf{secat}}
\DeclareMathOperator{\TCD}{\mathsf{TC}^{\mathcal{D}}}

\theoremstyle{plain}
\newtheorem{theorem}{Theorem}[section]
\newtheorem{prop}[theorem]{Proposition}
\newtheorem{lemma}[theorem]{Lemma}
\newtheorem{cor}[theorem]{Corollary}  
\newtheorem{theo}{Theorem}

\theoremstyle{definition}
\newtheorem{definition}[theorem]{Definition}

\theoremstyle{remark}
\newtheorem{remark}[theorem]{Remark}
\newtheorem{example}[theorem]{Example}

\begin{document}
\setlength{\parindent}{0cm}

\title{Oriented robot motion planning in Riemannian manifolds}

\author{Stephan Mescher}
\address{Mathematisches Institut \\ Universit\"at Leipzig \\ Augustusplatz 10 \\ 04109 Leipzig \\ Germany}
\email{mescher@math.uni-leipzig.de}

\date{\today}

\begin{abstract}
We consider the problem of robot motion planning in an oriented Riemannian manifold as a topological motion planning problem in its oriented frame bundle. For this purpose, we study the topological complexity of oriented frame bundles, derive an upper bound for this invariant and certain lower bounds from cup length computations. In particular, we show that for large classes of oriented manifolds, e.g. for spin manifolds, the topological complexity of the oriented frame bundle is bounded from below by the dimension of the base manifold.
\end{abstract}

\maketitle 	

\section{Introduction}

Based on the presentation of the robot motion planning problem by J.-C. Latombe in \cite{Latombe}, M. Farber has abstractified the search for motion planning problems to topological spaces, which we briefly want to recall. We want to model the situation of an autonomous robot moving in a specified workspace, e.g. a warehouse or a grid network. The workspace is modelled as a topological space $X$ while we imagine the robot as a point in its workspace, i.e. we ignore its specific shape. We want the robot to move autonomously, i.e. if it is located at a certain position $x$ and has to move to a different position $y$, the robot is supposed to decide autonomously which way to take from $x$ to $y$. In the abstract topological setting, this problem reads as follows. \\

\textbf{Topological motion planning problem:} \enskip Let $X$ be a topological space. For any two points $x,y \in X$ find a map $\gamma \in PX=C^0([0,1],X)$ with $\gamma(0)=x$ and $\gamma(1)=y$. \\

The choices of paths that a robot will make are encoded in a \emph{motion planning algorithm}. In the topological framework, a motion planning algorithm is seen as a map $s: X \times X \to PX$ with $(s(x,y))(0)=x$ and $(s(x,y))(1)=y$ for all $x,y \in X$. \\

Based on this framework, Farber has introduced his notion of \emph{topological complexity} in \cite{FarberTC}. Roughly, the topological complexity $\TC(X)$ is the minimal number $k \in \NN$ for which there exists a motion planning algorithm $s:X \times X \to PX$ for which $X \times X$ can be decomposed into $k$ pairwise disjoint subsets, such that $s$ is continuous on each of these subsets. We will give a more formal definition in Section 2 below. For surveys on topological complexity and related notions see \cite{FarberSurveyTC}, \cite[Chapter 4]{FarberBook} and \cite{FarberTCrecent}. \\

There are two essential differences between the robot motion planning problem as formulated  in \cite{Latombe} and its topological abstraction. Instead of general topological spaces, Latombe considers subsets of Euclidean space as workspaces, which suffices to model many practical situations, like the groundplan of a warehouse as a subset of $\RR^2$ or the interior of a warehouse as a subset of $\RR^3$. Since the actual robot  has a specific shape one loses a lot of information by simply modelling the robot as a point in its workspace as in the topological motion planning problem. Latombe overcomes this problem by still modelling the robot as a point (e.g. its center of mass), but equipping the point with a positive orthonomal basis (also called \emph{vielbein} in mathematical physics) which indicates the relative orientation of the robot in its workspace. In this situation we can formulate a motion planning problem including orientations. In other words, we do not only want the robot to find a path from one point to another, but also to find a way of rotating itself into a prescribed relative orientation. An appropriate mathematical description is given as follows. \\

\textbf{Linear oriented motion planning problem:} \enskip Let $U \subset \RR^n$. For any two points $x,y \in U$ and any two positive orthonormal bases $B_1,B_2 \in SO(n)$ find curves $\gamma \in PU$ and $\alpha \in P(SO(n))$ with $(\gamma(0),\alpha(0))=(x,B_1)$ and $(\gamma(1),\alpha(1))=(y,B_2)$. \\

As an intermediate step between the full abstraction to topological spaces and the situation of open subsets of Euclidean space, it is reasonable to assume that the robot's workspace is described as a smooth manifold. Manifolds are the natural choices to model robots moving in Euclidean space with kinematic constraints, see e.g. \cite[Chapter 1]{FarberBook} for a description of workspaces of robot arm linkages as submanifolds of Euclidean space. While kinematic constraints are discussed in \cite{Latombe} as restrictions on possible paths in the workspace after the motion planning problem has been formulated, we suggest to incorporate the constraints into the model for motion planning from the beginning by considering oriented Riemannian manifolds as workspaces. \\



The Riemannian metric will account for questions of lengths and deliver a notion of orthonormal bases of tangent spaces. The suitable adaptation of the positive orthonomal bases of Euclidean space in the linear oriented motion planning problem is the consideration of positive orthonormal (local) \emph{frames} of the tangent bundle of the manifold, i.e. families of bases of tangent spaces that move continuously along the tangent spaces of a curve in the manifold. The family of all positive orthonormal frames of tangent spaces of an oriented Riemannian manifold  $(M,g)$ are known to form a principal $SO(n)$-bundle over $M$, which we will denote by $F(M,g)$ and define in detail in Section \ref{SecBasic}. With respect to this bundle, the motion planning problem that is underlying this article can be made precise. \\

\textbf{Oriented motion planning problem:} \enskip Let $(M,g)$ be an oriented Riemannian manifold.  For any two points $x,y \in M$, any positive orthonormal bases $B_1$ of $T_xM$ and $B_2$ of $T_yM$ find $\gamma \in P(F(M,g))$ with $\gamma(0)=(x,B_1)$ and $\gamma(1)=(y,B_2)$. \\

One sees that the oriented motion planning problem in $(M,g)$ is nothing but the topological motion planning problem in its frame bundle $F(M,g)$. To investigate the complexity of oriented motion planning in $(M,g)$, it is thus required to study the topological complexity of motion planning in $F(M,g)$.  

Surprisingly, the complexity of the topological and the oriented motion planning problem might show a fundamentally different behaviour, as the following basic examples show.

\begin{example}
\label{ExBasic}
\begin{enumerate}[a)]
\item Let $n \in \NN$. By elementary methods, see \cite[Example 4.8]{FarberBook}, one shows that
$$\TC(S^n) = \begin{cases}
2 & \text{if $n$ is odd,} \\
3 & \text{if $n$ is even.}
\end{cases}$$
In particular, $\TC(S^n)$ takes one of only two values and depends only on the parity of $n$. 

Let $S^n$ be equipped with the round metric $g_n$ of constant curvature 1. The total space of the frame bundle of $(S^n,g_n)$ is given by $F(S^n,g_n) = SO(n+1)$ with the bundle projection $F(S^n,g_n) \to S^n$ being the projection onto the first column. Applying \cite[Lemma 8.2]{FarberInstab}, it follows that
$$\TC(F(S^n,g_n)) = \TC(SO(n+1)) = \cat(SO(n+1)),$$
the Lusternik-Schnirelmann category of $SO(n+1)$. As a consequence of \cite[Corollary 1.1]{KorbasSO2}, it holds that $\cat(SO(n+1)) \geq n+1$ for every $n \in \NN$, which implies that the sequence $(\TC(F(S^n,g_n)))_{n \in \NN}$ diverges.
\item Since $\TC$ is a homotopy invariant, see \cite[Theorem 3]{FarberTC}, it holds that $\TC(\RR^n)=1$ for every $n \in \NN$. With $g_n$ being the Euclidean metric on $\RR^n$ we further obtain
$$\TC(F(\RR^n,g_n)) = \TC(\RR^n \times SO(n)) = \TC(SO(n))=\cat(SO(n)) \geq n.$$
Thus, $(\TC(\RR^n))_{n \in \NN}$ is constant while $(\TC(F(\RR^n,g_n)))_{n \in \NN}$ diverges. 
\end{enumerate}
\end{example}

The following main result of this article, which summarizes several statements proven below, shows that this phenomenon occurs for a big class of Riemannian manifolds.

\begin{theo}
\label{Theo1}
Let $(M,g)$ be an oriented Riemannian manifold. If one of the following conditions holds:
\begin{enumerate}[(i)]
\item $M$ is a spin manifold,
\item $F(M,g) \to M$ is totally non-cohomologous to zero with respect to a field $\K$,
\end{enumerate}
then 
$$\TC(F(M,g)) \geq \dim M.$$
\end{theo}

In addition to this lower bound, we will also derive an upper bound for $\TC(F(M))$. A standard upper bound for topological complexity is given by $\TC(X)\leq 2 \dim X +1$ for any CW complex $X$, see \cite[Theorem 4]{FarberTC}. It was observed by M. Grant in \cite{GrantTCfibr} that the existence of Lie group actions whose fixed point sets have certain properties yield better upper bounds on $\TC(X)$ if $X$ is a closed manifold, which applies in particular to free Lie group actions. As a principal $SO(n)$-bundle, every bundle of positive orthonormal frames of an $n$-dimensional manifold is equipped with a free $SO(n)$-action. Making use of this fact and combining it with possible additional group actions on $M$, we will show the following.

\begin{theo}
\label{Theo2}
Let $(M,g)$ be an $n$-dimensional oriented Riemannian manifold. Let $G$ be a compact Lie group which acts freely and smoothly on $M$. Then 
$$\TC(F(M,g)) \leq \frac{n(n+3)}2 - \dim G +1.$$
\end{theo}

Note that this upper bound improves the standard upper bound roughly by a factor of $\frac12$. \\

In Section \ref{SecBasic} we recall basic notions regarding topological complexity and frame bundles as well as an elementary property of frame bundles under varying Riemannian metrics. We derive an upper bound for the topological complexity of oriented frame bundles in Section \ref{SecUpper}. Sections \ref{SecTNCZ} and \ref{SecSpin} are devoted to lower bounds on the topological complexity of frame bundles, for the case of totally non-cohomologous to zero frame bundle in the former section and on frame bundles of spin manifolds in the latter. We conclude the article by giving computations and estimates for certain basic classes of examples in Section \ref{SecEx}. \\

\emph{Throughout this article, all manifolds are assumed to be connected.}

\section*{Acknowledgements}

The author thanks John Oprea for helpful comments on an earlier draft of the manuscript and the anonymous referee for his remarks that helped improve the clarity of the exposition.

\section{Basic notions}
\label{SecBasic}
	
The notion of topological complexity is obtained as a special case of the notion of the sectional category of a fibration. We will give a brief definition of sectional category and of topological complexity and refer to \cite[Section 9.3]{CLOT} for details. 
	
\begin{definition}
\begin{enumerate}
\item Let $f: E \to B$ be a fibration. The \emph{sectional category} or \emph{Schwarz genus of $f$} is the smallest number $k \in \NN$ such that there exist open subsets $U_1,\dots,U_k \subset B$ with $\bigcup_{j=1}^k U_j=B$ and continuous maps $s_j: U_j \to E$ with $f \circ s_j = \mathrm{incl}_{U_j}$ for all $j \in \{1,2,\dots,k\}$. We denote the sectional category of $f$ by $\secat(f: E \to B)$, or simply $\secat(f)$. If there are no such sets and maps for any $k \in \NN$, we put $\secat(f:E\to B) := +\infty$.
\item Let $X$ be a topological space, $PX:= C^0([0,1],X)$ and $\pi: PX \to X \times X$, $\gamma \mapsto (\gamma(0),\gamma(1))$. The \emph{topological complexity of $X$} is given by 
$$\TC(X) := \secat(\pi: PX \to X \times X).$$
\item Given a subset $A \subset X \times X$, a \emph{motion planning algorithm over A} is a map $s: A \to PX$ with $ \pi \circ s = \mathrm{incl}_A$.
\end{enumerate}	
\end{definition}
 
It was shown by Farber in \cite{FarberTC} that $\TC(X)$ depends only on the homotopy type of $X$ and that 
$$\cat(X) \leq \TC(X) \leq \cat(X\times X),$$
where $\cat$ denotes the Lusternik-Schnirelmann category.

\begin{remark}
$\TC(X)$ is the minimal number of open domains required to cover all of $X \times X$ with continuous motion planning algorithms. From the viewpoint of robot motion planning, it is more convenient to consider motion planning algorithms $s: X \times X \to PX$ that restrict to continuous maps over a pairwise disjoint family of subsets of $X \times X$. It is shown in \cite[Section 4.2]{FarberBook} that if $X$ is a Euclidean neighborhood retract, e.g. a manifold, then indeed $\TC(X)$ is the smallest $k \in \NN$, such that there are $A_1,\dots,A_k \subset X \times X$ with $\bigcup_{j=1}^kA_j = X \times X$ and $A_i \cap A_j = \emptyset$ whenever $i \neq j$, such that there exists a continuous motion planning algorithm over each $A_j$. 
\end{remark}	

In this article, we want to study the topological complexity of orthonormal frame bundles. Before we proceed, we will present their formal definition and an important property.

\begin{definition}
Let $(M,g)$ be an $n$-dimenisonal oriented Riemannian manifold. Given $x \in M$ we let
$$F_x(M,g) := \left\{(v_1,\dots,v_n) \in (T_xM)^n \ | \ (v_1,\dots,v_n) \text{ is a positive orthonormal basis of } (T_xM,g_x) \right\}$$
and put 
$$F(M,g) := \{ (x,b) \ | \ x \in M, \ b \in F_x(M,g)\}.$$
Let $p: F(M,g) \to M$, $(x,b) \mapsto x$.  Then $p$ has the structure of a smooth principal $SO(n)$-bundle and is called \emph{the bundle of positive orthonormal frames of $(M,g)$}. 
\end{definition}

\begin{remark}
Given an $n$-dimensional oriented manifold $(M,g)$, it is a well-established fact from differential geometry that $F(M,g)$ is a trivial $SO(n)$-bundle if and only if $M$ is parallelizable, i.e. if there exist $n$ smooth vector fields which are fiberwise linearly independent over each point in $M$. 
\end{remark}

Occasionally, we will simply call $F(M,g)$ \emph{the oriented frame bundle of $(M,g)$}. A useful observation used throughout this article is that the topological complexity of a frame bundle is independent of the choice of Riemannian metric used to define it. This is a consequence of the following result.

\begin{prop}
Let $M$ be an $n$-dimensional orientable Riemannian manifold and let $g_0$ and $g_1$ be Riemannian metrics on $M$. Then $F(M,g_0)$ and $F(M,g_1)$ are isomorphic as principal $SO(n)$-bundles. 
\end{prop}
\begin{proof}
Let $GL_+(n,\RR)=\{A \in GL(n,\RR) \ | \ \det A >0\}$ and consider the principal $GL_+(n,\RR)$-bundle
$$GL_+(M) := \left\{(x,v_1,\dots,v_n) \ | \ x \in M, \ (v_1,\dots,v_n) \text{ positive basis of } T_xM \right\}.$$
The set of Riemannian metrics is well known to be convex, i.e. we obtain a family of Riemannian metrics $(g_t)_{t \in [0,1]}$ by putting
$$g_t = (1-t)g_0 + t g_1 \quad \forall t \in [0,1].$$
Let $p: [0,1] \times M  \to M$ denote the projection onto the second factor. Then $(g_t)_{t \in [0,1]}$ defines a Riemannian metric on $p^*TM$ and we let $E$ denote the $SO(n)$-bundle of orthonormal frames with respect to that metric. For $t \in [0,1]$ we define $i_t: M \to [0,1] \times M$, $i_t(x)=(t,x)$. By definition of the metric, one easily checks that
$$F(M,g_0)=i_0^*E, \quad F(M,g_1)=i_1^*E. $$
Since the maps $i_t$ define a homotopy from $i_0$ and $i_1$, a standard result for pullbacks of principal bundles, see \cite[Theorem I.11.5]{Steenrod}, implies that $F(M,g_0) \cong F(M,g_1)$.
\end{proof}

Thus, if we make no particular use of the choice of Riemannian metric, we will simply write $F(M)$ instead of $F(M,g)$ and refer to the oriented frame bundle of an arbitrary metric on $M$.

\section{Upper bounds and isometry groups}
\label{SecUpper}

As a basic and useful upper bound for topological complexity, it was shown by M. Farber in \cite[Theorem 5.2]{FarberInstab} that, given an $r$-connected CW complex $X$, $r \in \NN_0$, it holds that 
$$ \TC(X) \leq \frac{2 \dim X+1}{r+1}.$$
Since the fundamental group of the fiber of a positive orthonormal frame bundle is always non-vanishing, $F(M)$ will not be simply connected in the general case, such that Farber's bound only yields
$$\TC(F(M))\leq 2 \dim F(M)+1 = 2\dim M + 2\dim SO(n)+1= n(n+1)+1$$
for an $n$-dimensional oriented Riemannian manifold $M$.

If $M$ is parallelizable, i.e. if $F(M)$ is a trivial $SO(n)$-bundle, then we may apply results about the topological complexity of products and topological groups to obtain another upper bound. 

\begin{prop}
\label{PropParaUpper}
Let $(M,g)$ be an oriented $n$-dimensional Riemannian manifold. If $M$ is parallelizable, then 
$$\TC(F(M,g)) \leq \cat(SO(n))+ \TC(M)-1.$$
\end{prop}
\begin{proof}
If $M$ is parallelizable, then $F(M,g)$ is a trivial principal $SO(n)$-bundle. Hence, 
$$\TC(F(M,g))= \TC(SO(n) \times M) \leq \TC(SO(n)) + \TC(M) -1 = \cat(SO(n)) + \TC(M)-1,$$
where we have used \cite[Lemma 8.2]{FarberInstab} and \cite[Theorem 11]{FarberTC}.
\end{proof}
 
\begin{cor}
\label{CorLie}
Let $G$ be a Lie group and let $g$ be a Riemannian metric on $G$. Then 
$$\TC(F(G,g))\leq \cat(SO(n))+\cat G-1.$$
\end{cor}

The Lusternik-Schnirelmann category of $SO(n)$ will be discussed in greater detail in the upcoming section. 
\bigskip 

We want to establish a more general upper bound which makes use of relations between smooth group actions on manifolds and topological complexity. More precisely, we will prove the following upper bound by reducing it to an application of a more general result of M. Grant on smooth group actions, see \cite[Theorem 5.2]{GrantTCfibr}.

\begin{theorem}
\label{TheoremMainUpper}
Let $M$ be an oriented $n$-dimensional manifold and let $G$ be a compact Lie group which acts smoothly, orientation-preservingly and freely on $M$. Then
$$\TC(F(M)) \leq \frac{n(n+3)}2 -  \dim G+1.$$
\end{theorem}
\begin{proof}
Let $g$ be a $G$-invariant metric on $M$ and identify $G$ with a subgroup of $\Isom_+(M,g)$, the group of orientation-preserving isometries of $(M,g)$. The smooth group action $\Isom_+(M,g) \times M \to M$, $(\phi,x)\mapsto \phi(x)$, induces a smooth action of $\Isom_+(M,g)$ on $F(M,g)$ by 
$$\Isom_+(M,g) \times F(M,g) \to F(M,g), \quad \phi\cdot (x,b_1,\dots,b_n):= (\phi(x),D\phi_x(b_1),\dots,D\phi_x(b_n)).$$
By abuse of notation, we will continue writing $D\phi_x(b):= (D\phi_x(b_1),\dots,D\phi_x(b_n))$ for $x \in M$, $b=(b_1,\dots,b_n) \in F_x(M,g)$ and $\phi \in \Isom_+(M,g)$.

We further consider the free and transitive fiberwise right action $F(M,g) \times SO(n) \to F(M,g)$ given by the $SO(n)$-bundle structure. Passing to local coordinates, the associativity of matrix multiplication in the fibers shows that 
$$(\phi \cdot x) \cdot A = \phi \cdot (x \cdot A) \qquad \forall x\in F(M,g), \ \phi \in \Isom_+(M,g), \ A \in SO(n).$$
Thus, there is a smooth left $(\Isom_+(M,g)\times SO(n))$-action
$$ \Isom_+(M,g) \times SO(n) \times F(M,g) \to F(M,g), \quad (\phi,A)\cdot x = \phi \cdot x \cdot A^{-1}.$$
We consider the restriction of this action to $G \times SO(n)$. Let $(\phi,A) \in G \times SO(n)$ and $(x,b) \in F(M,g)$ with $(\phi, A) \cdot (x,b) = (x,b)$. Then it holds in particular that $\phi(x)=x$, which by the freeness of the $G$-action implies $\phi=\id_M$. But since the right-action of $SO(n)$ on $F(M,g)$ is free, the equation $(\id_M,A) \cdot (x,b)= (x,b)\cdot A^{-1}$ is only satisfied for $A = I_n$, the rank $n$ unit matrix. Thus, the compact Lie group $G \times SO(n)$ acts freely and smoothly on $F(M,g)$.
Hence, \cite[Corollary 5.3]{GrantTCfibr} implies
\begin{align*}
\TC(F(M,g)) &\leq 2 \dim F(M,g) - \dim (G \times SO(n)) +1  \\
&\leq 2 (\dim M + \dim SO(n)) - \dim G - \dim SO(n)+1 \\
&= 2\dim M + \dim SO(n)- \dim G+1= 2n + \frac{n(n-1)}{2} -\dim G+1 \\
&= \frac{n(n+3)}{2} - \dim G+1.
\end{align*}
\end{proof}

\begin{cor}
\label{CorSimpledim}
For every oriented $n$-dimensional manifold $M$ it holds that
$$\TC(F(M))\leq \frac{n(n+3)}{2}+1.$$
\end{cor}


\section{Lower bounds for TNCZ frame bundles}
\label{SecTNCZ}
This section is devoted to the proof of part (ii) of Theorem \ref{Theo1} from the introduction. Given a commutative ring $R$, a topological space $X$ and an ideal $I \subset \widetilde{H}^*(X;R)$ we let
$$\cl_R(I) := \sup \{ n \in \NN \ | \ \exists u_1,\dots,u_n \in I \ \text{s.t.} \ u_1 \cup \dots \cup u_n \neq 0\}.$$
Here, $\widetilde{H}^*$ denotes reduced singular cohomology. We let 
$$Z(X;R):=\ker \left[\Delta^*:\widetilde{H}^*(X\times X;R) \to \widetilde{H}^*(X;R) \right]$$ 
denote the set of zero-divisors of $X$ with coefficients in $R$. Note that every cohomology class $u \in \widetilde{H}^*(X;R)$ there is an associated zero-divisor
$$\bar{u} := 1 \times u - u \times 1 \in Z(X;R)$$
with $\times$ denoting the cohomology cross product. We put
$$\zcl_R(X) := \cl_R(Z(X;R)).$$
It was shown by Farber in \cite[Theorem 7]{FarberTC} that
\begin{equation}
\label{TCup}
\TC(X) \geq \zcl_R(X)+1
\end{equation}
for every commutative ring $R$. This inequality has an immediate consequence for frame bundles of parallelizable manifolds. 

\begin{prop}
\label{PropParaLower}
Let $\K$ be a field. If $M$ is a parallelizable manifold, then 
$$\TC(F(M)) \geq \zcl_{\K}(SO(n))+\zcl_{\K}(M)+1.$$
\end{prop}
\begin{proof}
Since $M$ is parallelizable and $\K$ is a field, the K\"unneth theorem implies that 
$$H^*(F(M)\times F(M);\K) \cong H^*(SO(n) \times SO(n);\K)\otimes_{\K} H^*(M\times M;\K)$$
as rings. One easily checks that this isomorphism restricts to an isomorphism
$$Z(F(M);\K) \cong Z(SO(n);\K)\otimes_{\K} Z(M;\K),$$
which apparently implies
$$\zcl_{\K}(F(M))= \zcl_{\K}(SO(n))+\zcl_{\K}(M).$$
The claim then follows from \eqref{TCup}.
\end{proof}

We recall that, given a commutative ring $R$, a fiber bundle $p:E \to B$ is \emph{totally non-cohomologous to zero (TNCZ) with respect to $R$} if the inclusion of a fiber $i: F \hookrightarrow E$ induces a surjective map $i^*:H^*(E;R) \to H^*(F;R)$. \\

The following statement is similar to and will be proven along the same lines as the unnumbered Lemma on p. 25 of \cite{HoraKorbas}. 

\begin{prop}
\label{Propzcl}
Let $p:E \to B$ be a fiber bundle, let $i: F \hookrightarrow E$ be the inclusion of a fiber and let $\K$ be a field. If $F$ is compact and $p$ is TNCZ with respect to $\K$, then 
$$\zcl_{\K}(E) \geq \zcl_{\K}'(F,E)+\zcl_{\K}(B),$$
where 
$$\zcl_{\K}'(F,E) = \sup \left\{n \in \NN \ | \ \exists u_1,\dots,u_n \in (i \times i)^*(Z(E;\K)) \ \text{s.t.} \ u_1 \cup \dots \cup u_n \neq 0 \right\}.$$
\end{prop}
\begin{remark}
In the situation of Proposition \ref{Propzcl}, the surjectivity of $i^*$ implies that
\begin{equation}
\label{zclbar}
\zcl'_{\K}(F,E) \geq \zcl_{\K}''(F):= \sup \{n \in \NN\ | \ \exists u_1,\dots,u_n \in \widetilde{H}^*(F;\K) \ \text{s.t.} \ \bar{u}_1\cup \dots \bar{u}_n \neq 0\},
\end{equation}
where $\bar{u} = 1 \times u - u \times 1$ for every $u \in \widetilde{H}^*(F;\K)$. This inequality provides a lower bound for $\zcl'_{\K}(F,E)$ that is independent of the total space of the bundle.
\end{remark}

\begin{proof}[Proof of Proposition \ref{Propzcl}]
Since $\K$ is a field, the assumption and the K\"unneth theorem imply that $(i\times i)^*:H^*(E \times E;\K) \to H^*(F \times F;\K)$ is surjective. Let $r:= \zcl'_{\K}(F,E)$. Then there are $\K$-linearly independent elements $a_1,\dots,a_r \in (i\times i)^*(Z(E;\K))$ with $a_1 \cup \dots \cup a_r \neq 0$. We pick $\alpha_1,\dots,\alpha_r \in Z(E;\K)$ with $(i \times i)^*(\alpha_j)=a_j$ for every $j \in \{1,\dots,r\}$. Then 
$$(i\times i)^*(\alpha_1 \cup \dots \cup \alpha_r) = (i\times i)^*(\alpha_1) \cup \dots \cup (i\times i)^*(\alpha_r) = a_1 \cup \dots \cup a_r \neq 0,$$
which implies $\alpha_1 \cup \dots \cup \alpha_r \neq 0$. 

By assumption on $F$, the fiber bundle $p\times p:E\times E \to B\times B$ satisfies the assumptions of the Leray-Hirsch theorem, see \cite[Theorem 4D.1]{Hatcher}, which implies that there is an isomorphism of $H^*(B \times B;\K)$-modules 
$$\Phi: H^*(F\times F;\K)\otimes_{\K} H^*(B\times B;\K) \to H^*(E\times E;\K). $$
By choosing a suitable $\K$-basis of $H^*(F\times F;\K)$ containing $a_1\cup \dots\cup a_r$, this isomorphism can be constructed such that 
$$\Phi((a_1 \cup \dots \cup a_r) \otimes \sigma) = \alpha_1\cup \dots \cup \alpha_r \cup (p\times p)^*(\sigma) \qquad \forall \sigma \in H^*(B \times B;\K).$$
Here, we consider the $H^*(B\times B;\K)$-module structures given by the diagonal $H^*(B \times B;\K)$-action of the standard action on the first and the trivial action on the second factor of $H^*(B\times B;\K)\otimes_{\K} H^*(F\times F;\K)$ and the action induced by $p^*:H^*(B\times B;\K) \to H^*(E\times E;\K)$ on $H^*(E\times E;\K)$. 


Let $s := \zcl_{\K}(B)$ and let $b_1,\dots,b_s \in Z(B;\K)$ with $b_1 \cup \dots \cup b_s \neq 0$. It follows from Leray-Hirsch that $(p\times p)^*$ is injective, so if we put $\beta_j := (p \times p)^*(b_j)$ for every $j \in \{1,2,\dots,s\}$, it follows that
$$\beta_1 \cup \dots \cup \beta_s = (p\times p)^*(b_1) \cup \dots \cup (p \times p)^*(b_s) = (p\times p)^*(b_1\cup \dots \cup b_s) \neq 0.$$
Moreover, since pullbacks of zero-divisors are zero-divisors, it holds that $\beta_1,\dots,\beta_s \in Z(E;\K)$.
Consequently, 
\begin{align*}
(\alpha_1\cup \dots \cup \alpha_r) \cup (\beta_1 \cup \dots \cup \beta_s) &=(\alpha_1\cup \dots \cup \alpha_r) \cup (p \times p)^*(b_1\cup \dots \cup b_s) \\
&= \Phi((b_1\cup \dots \cup b_s)\otimes_{\K} (a_1\cup \dots \cup a_r))\neq 0,
\end{align*}
which implies that $\zcl_{\K}(E) \geq r+s$.
\end{proof}

\begin{cor}
\label{CorTClower}
Let $(M,g)$ be an oriented $n$-dimensional Riemannian manifold whose positive orthonormal frame bundle $p: F(M,g) \to M$ is TNCZ with respect to $\K$. Then 
$$\TC(F(M,g)) \geq \zcl'_{\K}(SO(n),F(M,g))+\zcl_{\K}(M)+1 \geq \zcl''_{\K}(SO(n))+ \zcl_{\K}(M)+1.$$
\end{cor}

The following inequality will be useful to estimate the right-hand side of Corollary \ref{CorTClower} from below. 

\begin{lemma}
\label{Lemmacl}
Let $X$ be a topological space and $R$ be a commutative ring. Then $\zcl''_R(X) \geq \cl_R(X)$. 
\end{lemma}
\begin{proof}
Let $r := \cl_R(X)$ and $u_1,\dots,u_r \in H^*(X;R)$ with $u_1 u_2 \dots u_r \neq 0$. Put $\bar{u}_i:= 1 \times u_i - u_i \times 1\in H^*(X \times X;R)$ for every $i \in \{1,2,\dots,r\}$. Then $\bar{u}_1 \bar{u}_2 \dots \bar{u}_r \neq 0$, since it contains the non-vanishing summand $1\times u_1u_2\dots u_r$ and thus $\zcl''_R(X) \geq r$.
\end{proof}

\begin{remark}
Since $ \zcl_{\K}''(M) \geq \cl_{\K}(M)$ in the situation of Corollary \ref{CorTClower}, it follows in particular that
\begin{equation}
\label{EqclSOn}
\TC(F(M,g)) \geq \cl_{\K}(SO(n))+\zcl_{\K}(M) + 1. 
\end{equation}
\end{remark}

We want to give more explicit lower bounds on $\TC(F(M,g))$ by estimating the values of $\zcl''_{\K}(SO(n))$ from below. For this purpose, we will distinguish between fields of characterstic two and others and derive lower bounds in both cases.  \\

Let $\K$ be a field whose characteristic is not two. The cohomology ring of $SO(n)$ with coefficients in $\K$ is for example computed in \cite[Corollary III.3.15]{MimuraToda} and is for any $n \in \NN$ given by
\begin{align*}
H^*(SO(2n+1);\K)&\cong \Lambda_{\K}[a_3,a_7,\dots,a_{4n-1}] , \\
H^*(SO(2n+2);\K)&\cong \Lambda_{\K}[a_3,a_7,\dots,a_{4n-1},a'_{2n+1}],
\end{align*}
where $\Lambda_{\K}$ denotes the exterior $\K$-algebra on the generators in the square brackets and where the degree of each generator is given by its index. In particular, this shows that $H^*(SO(n);\K)$ has $\lfloor \frac{n}{2}\rfloor$ generators for every $n \in \NN$, each of odd degree. 

In our cup length computations we will make use of the following simple statement.

\begin{lemma}
\label{LemmaCup}
Let $X$ be a topological space, $R$ be a commutative ring and $u \in H^*(X;R)$. Put $\bar{u} := 1 \times u - u \times 1 \in H^*(X \times X;R)$. If $u^2\neq 0$ or if $u$ is of even degree and $H^*(X \times X;R)$ has no additive $2$-torsion, then $\bar{u}^2 \neq 0$. 
\end{lemma}
\begin{proof}
This is obvious, since the grading of the cup product implies that
$$\bar{u}^2=(1 \times u - u \times 1)^2 = 1 \times u^2-(1+(-1)^{\deg u}) u \times u + u^2 \times 1.$$
If $u^2 \neq 0$, then the summand $1 \times u^2$ will be non-vanishing, hence $\bar{u}^2\neq 0$. If $u$ is of even degree and $H^*(X \times X;R)$ has no $2$-torsion, then the middle summand will read $-2 u \times u$ and will be non-vanishing, hence $\bar{u}^2 \neq 0$.
\end{proof}

\begin{prop} 
\label{PropSOclR}
Let $n \in \NN$ and let $\K$ be a field whose characteristic is not two. Then:
\begin{enumerate}[a)]
\item  $\zcl''_{\K}(SO(2n)) \geq \begin{cases}
 2n & \text{if $n$ is even,} \\
 2n-1 & \text{if $n$ is odd.}
\end{cases}. $
\item $\zcl''_{\K}(SO(2n+1))\geq \begin{cases}
 2n & \text{if $n$ is even,}\\
 2n-1 & \text{if $n$ is odd.}
\end{cases}.$
\end{enumerate}
\end{prop}
\begin{proof}
\begin{enumerate}[a)]
\item We rewrite the $\K$-cohomology ring of $SO(2n)$ as $H^*(SO(2n);\K) \cong \Lambda_{\K}[u_1,\dots,u_n]$ and put $\bar{u}_i=1\times u_i -u_i \times 1 \in H^*(SO(2n);\K)$ for every $i \in \{1,2,\dots,n\}$. Then $\bar{u}_1\dots\bar{u}_n \neq 0$, since it contains the nontrivial summand $1 \times u_1u_2\dots u_n$.

If $n$ is even, then $\bar{u}_1\dots\bar{u}_n$ is of even degree, hence 
$$\bar{u}_1^2\dots \bar{u}_n^2 = \pm (\bar{u}_1\dots\bar{u}_n)^2 \neq 0 $$
by Lemma \ref{LemmaCup}. It follows that $\zcl''_{\K}(SO(2n)) \geq 2n$. 

If $n$ is odd, then $\bar{u}_1\dots \bar{u}_{n-1}$ has even degree, such that $\bar{u}_1^2\dots \bar{u}_{n-1}^2\neq 0$ by Lemma \ref{LemmaCup}. Consequently, 
$$\bar{u}_1^2\dots \bar{u}_{n-1}^2 \bar{u}_n \neq 0.$$
Thus, $\zcl''_{\K}(SO(2n)) \geq 2(n-1)+1 = 2n-1$. 
\item This is shown along the lines of part a), since $H^*(SO(2n+1);\K)$ is an exterior $\K$-algebra on $n$ generators as well.
\end{enumerate}
\end{proof}

\begin{theorem}
\label{TheoremTCzclR}
Let $M$ be an oriented Riemannian manifold and assume that $F(M) \to M$ is TNCZ with respect to a field $\K$ whose characteristic is not two.
\begin{enumerate}[a)]
\item If $M$ is $2n$-dimensional, $n \in \NN$, then 
$$\TC(F(M)) \geq \begin{cases}
  \zcl_{\K}(M) + 2n+1 & \text{if $n$ is even,} \\
  \zcl_{\K}(M) + 2n & \text{if $n$ is odd.}
\end{cases}$$
\item If $M$ is $(2n+1)$-dimensional, $n \in \NN$, then 
$$\TC(F(M)) \geq \begin{cases}
  \zcl_{\K}(M) + 2n+1 & \text{if $n$ is even,} \\
  \zcl_{\K}(M) + 2n & \text{if $n$ is odd.}
\end{cases}$$
\end{enumerate}
\end{theorem}
\begin{proof}
This follows immediately from combining Corollary \ref{CorTClower} and Proposition \ref{PropSOclR}.
\end{proof}

\begin{cor}
\label{CorParaDiv}
Let $M$ be an oriented manifold and assume that $F(M) \to M$ is TNCZ with respect to a field $\K$ whose characteristic is not two. Then
$$\TC(F(M)) \geq \dim M.$$
\end{cor}
\begin{remark}
Corollary \ref{CorParaDiv} can be interpreted as another instance of the phenomenon already observed in Example \ref{ExBasic}. Here, $\TC(F(M))$ is again bounded from below by the dimension of $M$. 
\end{remark}

We turn our attention to cohomology with coefficients in a field $\F_2$ of characteristic two. The $\F_2$-cohomology ring of $SO(n)$ has a more sophisticated ring structure than its cohomology with other field coefficients, which will improve the lower bound on $\TC(F(M))$ derived above. This cohomology ring has a particular significance for our computations because of its relation to Lusternik-Schnirelmann category. The values of $\cat(SO(n))$ have been computed for $n \leq 5$ by I. James and W. Singhof in \cite{JamesSinghof}, for $n \leq 9$ by  N. Iwase, M. Mimura and T. Nishimoto in \cite{IwaseMimuraNishi} and for $n=10$ by Iwase, K. Kikuchi and T. Miyauchi in \cite{IwaseSO10}. It was shown that for each $n \leq 10$, it holds that 
\begin{equation}
\label{Eqcatcup}
\cat(SO(n)) = \cl_{\F_2}(SO(n))+1.
\end{equation}
It is an open question asked by Korba\v{s} in \cite[Question 1.1]{KorbasSO2} if the equality $\cat(SO(n)) = \cl_{\F_2}(SO(n))+1$ holds for all $n \in \NN$. If the answer is affirmative, then the requirement that $n \leq 10$ can be dropped in the following proposition.

\begin{prop}
Let $M$ be an $n$-dimensional oriented manifold, where $n \leq 10$. If $M$ is parallelizable and if $\TC(M)= \zcl_{\F_2}(M)+1$, then 
$$\TC(F(M)) = \cat(SO(n))+\TC(M).$$
\end{prop}
\begin{proof}
This follows directly from combining equation \eqref{Eqcatcup} with Propositions \ref{PropParaLower} and \ref{PropParaUpper}.
\end{proof}

It is a well-established computation, see \cite[Theorem 3D.2]{Hatcher} or \cite[Corollary 3.15]{MimuraToda}, that for each $n \in \NN$ there is an isomorphism of rings
$$H^*(SO(n);\F_2) \cong \bigotimes_{i \text{ odd},\  i < n} \F_2[\beta_i]/(\beta_i^{p_i}) ,$$
where $p_i = \inf \{ 2^k \ | \ k \in \NN, \ i2^k \geq n\}$ and $\beta_i$ has cohomological degree $i$ for all $i <n$. This formula obviously implies that
$$\cl_{\F_2}(SO(n)) = \sum_{i=1}^{\lfloor\frac{n}{2}\rfloor} (p_{2i-1}-1).$$
An explicit formula for $\cl_{\F_2}(SO(n))$ was obtained by J. Korba\v{s}. 
\begin{theorem}[{\cite[Theorem 1.1]{KorbasSO2}}]
\label{TheoremKorbas}
Let $n \in \NN$ and let $n-1$ have the dyadic expansion $n-1 = \sum_{i=0}^kn_i2^i$. Then
$$\cl_{\F_2}(SO(n)) = n-1 +\sum_{i=1}^k in_i 2^{i-1}.$$
\end{theorem}
Note that Korba\v{s} formulates his result only for the case $\F_2=\ZZ_2$ in \cite{KorbasSO2}, but his proof, which depends only on the combinatorial structure of the cohomology ring, generalizes to arbitrary fields of characteristic two. Computing the first values of this formula particulary provides
$$\cl_{\F_2}(SO(2))=1, \quad \cl_{\F_2}(SO(3))=3, \quad \cl_{\F_2}(SO(4))=4, \quad \cl_{\F_2}(SO(5))=8, \quad \dots$$ 

\begin{prop}
\label{PropParazcl}
Let $M$ be an oriented $n$-dimensional Riemannian manifold for which $F(M) \to M$ is TNCZ with respect to $\F_2$. Let $n-1$ have the dyadic expansion $n-1=\sum_{i=1}^k n_i2^{i-1}$. Then 
$$\TC(F(M)) \geq \zcl_{\F_2}(M) + n +\sum_{i=1}^k in_i 2^{i-1}.$$
\end{prop}
\begin{proof}
This follows immediately from inserting the equation from Theorem \ref{TheoremKorbas} into \eqref{EqclSOn}.
\end{proof}

\begin{remark}
The attentive reader might notice that we did not consider the question whether the value of $\zcl_{\F_2}(SO(n))$ exceeds the one of $\cl_{\F_2}(SO(n))$. More precisely, in the above notation we make use of the fact that $a_3^{p_3-1}a_7^{p_7-1}\dots a_{2m-1}^{p_{2m-1}-1} \neq 0 \in H^*(SO(n);\F_2)$, where $m:= \lfloor \frac{n}{2}\rfloor$, which implies that the associated zero-divisors satisfy $\bar{a}_3^{p_3-1}\bar{a}_7^{p_7-1}\dots \bar{a}_{2m-1}^{p_{2m-1}-1} \neq 0 \in H^*(SO(n)\times SO(n);\F_2)$. However, we have not checked whether there is an $i \in \{3,7,\dots,2m-1\}$ and $k_i \in \NN$, such that $\bar{a}^{k_i+p_i-1} \neq 0$.  While this might a priori be the case, one verifies in this situation that such numbers $k_3,\dots,k_{2m-1}$ do not exist. This follows from a long but straightforward computation using standard results on the parities of binomial coefficients.
\end{remark}

Put together, Corollary \ref{CorParaDiv} and Proposition \ref{PropParazcl} show the validity of part (ii) of Theorem \ref{Theo1} from the introduction.

\section{Lower bounds for frame bundles of spin manifolds}
\label{SecSpin}
In this section, we will prove part (i) of Theorem \ref{Theo1} from the introduction. The Serre exact sequence associated with the fiber bundle $F(M) \to M$ yields that the map $i^*:H^1(F(M);\ZZ_2) \to H^1(SO(n);\ZZ_2)$ is surjective if and only if $w_2(M)=0$, i.e. if the second Stiefel-Whitney class of $M$ vanishes, see \cite[Section II.1]{LawsonMich} for details. Hence, a necessary condition for $F(M) \to M$ to be TNCZ with respect to $\ZZ_2$ is that $M$ admits a spin structure. This observation motivates a study of the relations between the existence of spin structures and the topological complexity of frame bundles. \\


We recall that $\Spin(n)$ is the universal covering space of $SO(n)$ for $n \geq 3$ and that the corresponding covering map $p_n: \Spin(n) \to SO(n)$ is two-fold. An $n$-dimensional spin manifold, $n \geq 3$, is an oriented Riemannian manifold $(M,g)$ equipped with a principal $\Spin(n)$-bundle $\Spin(M) \to M$ and a two-fold covering $p:\Spin(M) \to F(M,g)$. $(\Spin(M),p)$ is called a \emph{spin structure} on $M$. For details on spin groups and spin structures, see \cite{LawsonMich}. 

\begin{definition}
Let $M$ be an $n$-dimensional spin manifold, $n \geq 3$. With respect to the action $\ZZ_2 \times \Spin(M) \to \Spin(M)$ by deck transformations we let 
$$\Sc(M) := \Spin(M) \times_{\ZZ_2} \Spin (M)$$ denote the orbit space of the diagonal action $\ZZ_2 \times \Spin(M) \times \Spin (M) \to \Spin(M) \times \Spin(M)$, $(g,x,y) \mapsto (gx,gy)$. By elementary covering theory, the four-fold covering $p \times p$ induces a two-fold regular covering 
$$q_M: \Sc(M) \to F(M) \times F(M).$$
Here, $F(M)$ is supposed to be taken with respect to the implicitly chosen metric. For every $n \in \NN$ with $n \geq 3$ we further put
$$\Sc_n := \Spin(n) \times_{\ZZ_2} \Spin(n),$$
which denotes again the orbit space with respect to the diagonal $\ZZ_2$-action. The four-fold covering $p_n \times p_n: \Spin(n) \times \Spin(n) \to SO(n) \times SO(n)$ induces a two-fold regular covering
$$q_n: \Sc_n \to SO(n) \times SO(n).$$
\end{definition}

Since $\ZZ_2$ acts fiberwise on $\Spin(M) \times \Spin(M)$ in the previous definition, it follows from elementary bundle theory, see \cite[Section I.7.4]{Steenrod} that $\Sc(M)$ is a smooth fiber bundle with fiber $\Sc_n$. 

The relevance of spin structures for our considerations of topological complexity is based on the following statement.

\begin{prop}
\label{PropTCsecatq}
Let $M$ be a spin manifold. Then 
$$\TC(F(M)) \geq \secat(q_M: \Sc(M) \to F(M,g) \times F(M,g)).$$
\end{prop}
\begin{proof}
This is a straightforward application of \cite[Theorem 4.1]{FTY}.
\end{proof}

\begin{cor}
\label{CorTCsecatpara}
Let $M$ be an $n$-dimensional oriented manifold. If $M$ is parallelizable, then 
$$\TC(F(M)) \geq \secat(q_n: \Sc_n \to SO(n) \times SO(n)).$$
\end{cor}
\begin{proof}
Let $F(M)$ be the orthonormal frame bundle with respect to an arbitrary Riemannian metric on $M$. Since $M$ is parallelizable, $F(M) \cong M \times SO(n)$ as principal $SO(n)$-bundles. Thus, the double cover 
$$\id_M \times p_n: M \times \Spin(n) \to M \times SO(n)$$
defines a spin structure on $M$, where we view $M \times \Spin(n)$ as the trivial principal $\Spin(n)$-bundle. Since the $\ZZ_2$-actions are defined fiberwise, one checks that with respect to this spin structure
$$\Sc(M)=(M \times \Spin(n))\times_{\ZZ_2}(M \times \Spin(n)) \cong M \times (\Spin(n)\times_{\ZZ_2} \Spin(n))=M \times \Sc_n$$
as fiber bundles and that under these identifications $q_M$ corresponds to $\id_M \times q_n$, which shows the claim, since obviously $\secat(\id_M \times q_n)= \secat(q_n)$.
\end{proof}

Based on the previous two statements, we want to establish lower bounds for $\TC(F(M))$ by deriving lower bounds for $\secat(q_M)$ using cohomology methods. The first of these bounds is a direct implication of a result by A. Schwarz.
\begin{prop}
\begin{enumerate}[a)]
\item Let $\K$ be a field and $k \in \NN$. If $$q_M^*:H^i(F(M,g)\times F(M,g);\K) \to H^i(\Sc(M);\K)$$ is surjective for all $i < 2k$, then $\TC(F(M)) \geq 2k+1$.
\item Let $n \in \NN$. If 
$$q_M^*:H^i(F(M,g)\times F(M,g);\ZZ_2) \to H^i(\Sc(M);\ZZ_2)$$ 
is surjective for all $i <n$, then $\TC(F(M))\geq n+1$.
\end{enumerate}
\end{prop}
\begin{proof}
We use the facts that 
$H^i(B\ZZ_2;\K)=H^i(\RP^\infty;\K) \neq 0$ whenever $i$ is even and that $H^i(\RP^\infty;\ZZ_2) \neq 0$ for all $i \in \NN$ and apply \cite[Theorem 17]{SchwarzGenus} to derive that $\secat(q_M) \geq 2k+1$ in the situation of part a) and $\secat(q_M)\geq n+1$ in the situation of part b). In both cases, the claim follows from Proposition \ref{PropTCsecatq}.
\end{proof}

Given a two-fold covering $p: \tilde{X} \to X$, we denote its \emph{characteristic class} or \emph{Stiefel-Whitney class} by $w(p) \in H^1(X;\ZZ_2)$, see \cite[Section 4.3.2]{Hausmann} or \cite[Section 8.2.4]{Kozlov}. We recall that $w(p) = f^*(\iota)$, where $f: X \to B\ZZ_2 = \RP^\infty$ is a classifying map for $p$ and $\iota \in H^1(\RP^\infty,\ZZ_2)$ denotes the generator. We further put
$$h(p) := \sup \{ k \in \NN_0 \ | \ w(p)^k=w(p) \cup \dots \cup w(p) \neq 0 \in H^k(X;\ZZ_2) \}.$$
This number was considered under the name Stiefel-Whitney height by D. Kozlov in \cite{KozlovTest}.

\begin{theorem}
\label{TheoremTCheight}
	Let $M$ be a spin manifold. Then $$\TC(F(M)) \geq h(q_M)+1. $$
\end{theorem}
\begin{proof}
	By \cite[Theorem 4]{SchwarzGenus}, it holds that 
	\begin{equation}
\label{clest}
	\secat(q_M) \geq \cl_{\ZZ_2}\left(\ker [q_M^*:\widetilde{H}^*(F(M)\times F(M);\ZZ_2) \to \widetilde{H}^*(\Sc(M);\ZZ_2) ]\right)+1.
	\end{equation}
	The map $q_M^*$ is part of the transfer exact sequence of the double cover $q_M$, see \cite[Section 4.3.3]{Hausmann}. This sequence reads as
$$ \dots \to H^{i-1}(F(M)^2;\ZZ_2) \stackrel{w(q_M)\cup\cdot}{\longrightarrow} H^{i}(F(M)^2;\ZZ_2) \stackrel{q_M^*}{\longrightarrow} H^i(\Sc(M);\ZZ_2) \stackrel{\text{tr}}{\longrightarrow} H^i(F(M)^2;\ZZ_2) \to \dots, $$
where $\mathrm{tr}$ denotes the transfer homomorphism of $q$. In particular,  the exactness of the sequence implies that 
$\ker q_M^* = \im (w(q_M) \cup \cdot)$. Combining this observation with Proposition \ref{PropTCsecatq} and the inequality \eqref{clest} yields
$$\TC(F(M)) \geq \cl_{\ZZ_2}\left(\im \left[w(q_M) \cup \cdot: H^*(F(M)^2;\ZZ_2) \to H^*(F(M)^2;\ZZ_2) \right] \right)+1. $$
This implies the claim, since
\begin{align*}
h(q_M)&=\sup \{k \in \NN_0 \ | \ w(q_M)^k \neq 0\} = \sup \{k \in \NN_0 \ | \ (w(q_M)\cup 1)^k \neq 0\} \\
&\leq \cl_{\ZZ_2}\left(\im \left[w(q_M) \cup \cdot: H^*(F(M)^2;\ZZ_2) \to H^*(F(M)^2;\ZZ_2) \right] \right).
\end{align*}
\end{proof}

\begin{cor}
Let $M$ be an $n$-dimensional oriented manifold. If $M$ is parallelizable, then
$$\TC(M) \geq h(q_n)+1.$$
\end{cor}
\begin{proof}
This follows from Corollary \ref{CorTCsecatpara} in the same way that Theorem \ref{TheoremTCheight} follows from Proposition \ref{PropTCsecatq}.
\end{proof}

In the following we consider the spaces $\Sc(M)$ and $\Sc_n$ as spaces with free $\ZZ_2$-actions given by the deck transformation actions of $q_M$ and $q_n$, resp.

\begin{definition}
\begin{enumerate}[a)]
\item Given a spin manifold $M$ we let $i(M)$ denote the biggest integer $k \in \NN$ for which there exists a continuous $\ZZ_2$-equivariant map
$$f: S^k \to \Sc(M)$$
with respect to the antipodal involution on $S^k$. 
\item For each $n \in \NN$ with $n \geq 3$  we let $i(n)$ denote the biggest integer $k \in \NN$ for which there exists a continuous $\ZZ_2$-equivariant map
$$f: S^k \to \Sc_n$$
with respect to the antipodal involution on $S^k$. 
\end{enumerate}
\end{definition}
\begin{remark}
The number $i(M)$ is a special case of the notion of the \emph{index} of a topological space with a free $\ZZ_2$-action that was introduced by P. Conner and E. Floyd in \cite{ConnerFloyd}.
\end{remark}

\begin{prop}
\label{PropIndex}
Let $M$ be an $n$-dimensional spin manifold. Then 
$$\TC(F(M)) \geq i(M)+1 \geq i(n)+1. $$
\end{prop}
\begin{proof}
By Theorem \ref{TheoremTCheight}, the first inequality follows if we can show that $h(q) \geq i(M)$. This inequality follows from a general property of characteristic classes of double covers, see \cite[Section 4.3.3]{Kozlov}.

Concerning the second inequality, we assume that there exists a continuous $\ZZ_2$-equivariant map $f:S^k \to \Sc_n$ for some $k \in \NN$. The inclusion of a fiber $j: \Sc_n \hookrightarrow \Sc(M)$ is obviously $\ZZ_2$-equivariant, such that $j \circ f:S^k\to \Sc(M)$ is continuous and $\ZZ_2$-equivariant. Thus, $i(n) \leq i(M)$. 
\end{proof}

\begin{prop}
\label{Propi}
\begin{enumerate}[a)]
\item The sequence $(i(n))_{n\geq 3}$ is monotonically increasing. 
\item For every $n \in \NN$ it holds that $i(n) \geq n-1$.
\end{enumerate}
\end{prop}
\begin{proof}
\begin{enumerate}[a)]
\item Let $n \in \NN$ with $n \geq 3$. The group $\Spin(n)$ is explicitly given a subset of the Clifford algebra $\Cl(n)$, where $\Cl(n)$ is given as follows. For all $i,k \in \NN$ with $i \leq k$ we let $e_i$ denote the $i$-th unit vector in $\RR^k$. Considering $(e_1,\dots,e_n)$ as an orthonormal basis of $\RR^n$, we have
$$\Cl(n) = \quot{\RR[e_1,\dots,e_n]}{(e_ie_j+e_je_i+2\delta_{ij}, \ i,j \in \{1,2,\dots,n\})}.$$
The group $\Spin(n)$ is then given by
$$\Spin(n) = \{v_1v_2\dots v_{2k} \in \Cl(n) \ | \ v_1,\dots,v_{2k} \in S^{n-1}, \ k \in \NN \}$$
with Clifford multiplication as group operation, see \cite[Section I.3]{LawsonMich} for details. Viewing $\ZZ_2\cong \{-1,1\}$ the free $\ZZ_2$-action defined by the two-fold covering $p_n: \Spin(n) \to SO(n)$ is then given by 
$$(-1) \cdot v_1v_2 \dots v_{2k} = (-v_1)v_2\dots v_{2k} \in \Spin(n) \quad \forall v_1v_2\dots v_{2k} \in \Spin(n).$$
The inclusion $\RR^n \to \RR^{n+1}$, $v \mapsto (v,0)$ obviously induces inclusions 
$$\bar{j}_n: \Cl(n) \to \Cl(n+1) \quad \text{and} \quad j_n:=\bar{j}_n|_{\Spin(n)}: \Spin(n) \to \Spin(n+1).$$
It is easy to see that $j_n$ is continuous and $\ZZ_2$-equivariant. Consequently, the product $j_n \times j_n: \Spin(n) \times \Spin(n) \to \Spin(n+1)\times \Spin(n+1)$ is $\ZZ_2 \times \ZZ_2$-equivariant. In particular, $j_n \times j_n$ induces a map 
$$f_n: \Sc_n= \Spin(n) \times_{\ZZ_2} \Spin(n) \to \Spin(n+1)\times_{\ZZ_2} \Spin(n+1)= \Sc_{n+1}, $$
which is again continuous. Since the diagonal subgroup is a normal subgroup of $\ZZ_2 \times \ZZ_2$, $f_n$ is $\ZZ_2$-equivariant with respect to the induced $\ZZ_2$-actions on $\Sc_n$ and $\Sc_{n+1}$, which correspond to the deck transformation actions of $q_n$ and $q_{n+1}$, respectively. Thus, we have shown that for each $n \geq 3$ there exists a continuous $\ZZ_2$-equivariant map $f_n: \Sc_n \to \Sc_{n+1}$. 
Let now $k := i(n)$ and let $\varphi: S^k \to \Sc_n$ be continuous and $\ZZ_2$-equivariant. Then $f_n \circ \varphi: S^k \to \Sc_{n+1}$ is continuous and $\ZZ_2$-equivariant as well, hence $i(n+1)\geq k=i(n)$. 
\item To find a continuous $\ZZ_2$-equivariant map $f: S^{n-1} \to \Sc_n$, it suffices to find maps $f_1,f_2: S^{n-1} \to \Spin(n)$ with $f_1(-v)=-f_1(v)$ and $f_2(-v)=f_2(v)$ for all $v \in S^{n-1}$. Given such maps, one checks without difficulties that $f := r \circ (f_1 \times f_2)$ has the desired properties, where $r: \Spin(n) \times \Spin(n) \to \Sc_n$ denotes the orbit space projection.
But if we put 
$$f_1(v) = ve_1 \ , \quad f_2(v) = s_0 \quad \forall v \in S^{n-1},$$
where $s_0 \in \Spin(n)$ arbitrary, these maps have the desired properties. Thus, $i(n) \geq n-1$.
\end{enumerate}
\end{proof}

\begin{cor}
\label{Corspin}
Let $M$ be a spin manifold. Then $\TC(F(M)) \geq \dim M$. 
\end{cor}
\begin{proof}
This follows direcly from combining Proposition \ref{PropIndex} with part b) of Proposition \ref{Propi}.
\end{proof}

\begin{remark}
Under certain connectivity assumptions, $\secat(q_M)$ is related to the $\mathcal{D}$-topological complexity as introduced by Farber, Grant, G. Lupton and J. Oprea in \cite{FGLObredon} and further studied by the authors in \cite{FGLOupper}. Assume that $M$ is simply connected. One derives from the long exact homotopy sequence of $\Spin(M) \to M$ that $\Spin(M)$ is simply connected as well, such that $\Spin(M)$ is the universal covering space of $F(M)$. 

Assume additionally that the connecting homomorphism $\pi_2(M) \to \pi_1(SO(n))$ from the long exact homotopy sequence of $F(M) \to M$ vanishes, which happens e.g. if $M$ is $2$-connected. Then the sequence yields $\pi_1(F(M)) = \pi_1(SO(n))=\ZZ_2$. In this case $\Sc(M)$ is the connected covering space of $F(M) \times F(M)$ associated with the diagonal subgroup of $\pi_1(F(M)\times F(M))$. 
This covering was studied in a more general setting in \cite{FGLOupper} (in their notation, $\Sc(M) = \widehat{F(M) \times F(M)}$). By \cite[Proposition 2.2]{FGLOupper}, it then holds that 
$$\TCD(F(M)) = \secat (q_M: \Sc(M) \to F(M) \times F(M)),$$
where $\TCD(F(M))$ denotes the $\mathcal{D}$-topological complexity of $F(M)$, see \cite[Definition 2.3.1]{FGLObredon} or \cite[Definition 2.1]{FGLOupper} for its definition. Since every $2$-connected manifold is spin, the above computations show that if $M$ is a closed $2$-connected manifold, then 
$$ \TCD(F(M)) \geq \dim M.$$
\end{remark}

Corollary \ref{Corspin} shows the validity of part (i) of Theorem \ref{Theo1} from the introduction.

\section{Examples}
\label{SecEx}

\subsection{Tori} 
\label{ExTori}
Let $T^n$ be the $n$-torus and let $F(T^n)$ denote the frame bundle of $T^n$ with respect to an arbitrary Riemannian metric. Since $T^n$ is a Lie group with $\cat(T^n)=n+1$, it follows from Corollary \ref{CorLie} that 
$$\TC(F(T^n)) \leq \cat (SO(n))+n. $$
Corollary \ref{CorTClower} and Lemma \ref{Lemmacl} further imply that
$$\TC(F(T^n)) \geq \cl_{\ZZ_2}(SO(n))+\cl_{\ZZ_2}(T^n)+1= \cl_{\ZZ_2}(SO(n))+n+1.$$
Hence, for $n \leq 10$, it holds that
$$\TC(F(T^n))= \cat(SO(n))+n+1.$$
Again, the assumption that $n \leq 10$ is only required to ensure $\cat(SO(n))=\cl_{\ZZ_2}(SO(n))+1$.

 \subsection{Closed oriented surfaces} 
 We have already seen in Example \ref{ExBasic} that 
$$\TC(F(S^2)) = \cat(SO(3))=\cat(\RR P^3) = 4 $$
and in Example \ref{ExTori} that
$$\TC(F(T^2)) = \cat(SO(2))+2=4.$$
Let $\Sigma_g$ be a  closed oriented surface of genus $g \geq 2$, equipped with an arbitrary Riemannian metric. Theorem \ref{TheoremMainUpper} delivers
$$\TC(F(\Sigma_g)) \leq \frac{2\cdot 5}{2}+1= 6.$$
It follows immediately from the Serre exact sequence of the bundle $F(\Sigma_g) \to \Sigma_g$ with $\ZZ_2$-coefficients and the fact that $w_2(\Sigma_g)=0$ that $i^*:H^1(\Sigma_g;\ZZ_2) \to H^1(SO(2);\ZZ_2)$ is surjective, where $i$ denote the inclusion of a fiber. Thus, $F(\Sigma_g) \to \Sigma_g$ is TNCZ with respect to $\ZZ_2$ and we obtain from Corollary \ref{CorTClower} that
$$\TC(F(\Sigma_g)) \geq \cl_{\ZZ_2}(SO(2))+\zcl_{\ZZ_2}(\Sigma_g)+1 \geq \zcl_{\ZZ_2}(\Sigma_g)+2.$$ 
There are $u,v \in H^1(\Sigma_g;\ZZ_2)$ with $u^2=v^2=0$ and $u v\neq 0$. Letting $\bar{u},\bar{v} \in H^1(\Sigma_g \times \Sigma_g;\ZZ_2)$ denote the associated zero-divisors, one computes that $\bar{u}\bar{v} = 1 \times uv -uv \times 1$,
from which it follows that 
$$\bar{u}^2\bar{v} = u \times uv-uv\times u \neq 0.$$
Hence $\zcl_{\ZZ_2}(\Sigma_g) \geq 3$, which shows $\TC(F(\Sigma_g))\geq 5$ and thus
$$\TC(F(\Sigma_g))\in \{5,6\} \qquad \forall g \geq 2.$$
In comparison with \cite[Proposition 4.43]{FarberBook}, we have shown in particular that
$$\TC(\Sigma_g) \leq \TC(F(\Sigma_g)) \leq \TC(\Sigma_g)+1 \qquad \forall g \in \NN_0,$$
i.e. that the minimal order of instability of solutions of the oriented motion planning problem either coincides with or is by one bigger than the one of the topological motion planning problem. 


\subsection{Three-dimensional manifolds}

Let $M$ be a three-dimensional oriented Riemannian manifold. Since every oriented $3$-manifold is parallelizable and since $\cl(SO(3))+1=\cat (SO(3))=\cat(\RP^3)=4$, we obtain from Proposition \ref{PropParaUpper} that
$$\TC(F(M)) \leq \TC(M) +3$$
and from Proposition \ref{PropParaLower} that
$$\TC(F(M)) \geq \zcl_{\ZZ_2}(M)+4 \quad \text{and} \quad \TC(F(M)) \geq \zcl_{\RR}(M)+2.$$
Since $\TC(M)\leq 2\dim M +1=7$, it follows that
$$5 \leq \TC(F(M)) \leq 10. $$
It is further shown in \cite[Proposition 4.5]{ORdetect} that $\cl_{\ZZ_2}(M)=3$ if $M$ is irreducible and non-aspherical and $\pi_1(M)$ is infinite. Thus, in this case we obtain
$$7 \leq \TC(F(M)) \leq 10,$$
since $\zcl_{\ZZ_2}(M) \geq \cl_{\ZZ_2}(M)=3$.


\subsection{Complex projective spaces}

Let $\CP^n$ be equipped with a Riemannian metric. Theorem \ref{TheoremMainUpper} implies that
$$\TC(F(\CP^n)) \leq \frac{2n(2n+3)}{2}+1 = n(2n+3)+1.$$
It is shown in \cite[Section 9.25]{GHV3} that $F(\CP^n) \to \CP^n$ is TNCZ with respect to $\RR$. The real cohomology ring of $\CP^n$ is given by $H^*(\CP^n;\RR) \cong \RR[u]/(u^{n+1})$, where $u$ has degree $2$. One computes that $\bar{u}= 1 \times u - u \times 1 \in H^2(\CP^n\times\CP^n;\RR)$ satisfies
$$\bar{u}^{2n} = (-1)^n \binom{2n}{n} u^n \times u^n \neq 0.$$
Hence, $\zcl_{\RR}(\CP^n)\geq 2n$ and Theorem \ref{TheoremTCzclR} yields 
$$\TC(F(\CP^n))\geq \begin{cases}
4n+1& \text{if $n$ is even,} \\
4n & \text{if $n$ is odd.} 
\end{cases} $$
In particular, for $n=2$ and $n=3$ we obtain
$$ 9 \leq \TC(F(\CP^2))\leq 15, \quad 12 \leq \TC(F(\CP^3)) \leq 28. $$


\subsection{Parallelizable real projective spaces}
We recall that the real projective space $\RP^n$ is parallelizable if and only if $n \in \{1,3,7\}$. In \cite{FTY} it was shown that in these three cases, $\TC(\RP^n)=n+1$. Thus, Proposition \ref{PropParaUpper} yields
$$\TC(F(\RP^n)) \leq \cat (SO(n)) +  n \quad \text{for } n \in \{1,3,7\}.$$
On the other hand, Proposition \ref{PropParaLower} and Lemma \ref{Lemmacl} imply 
$$\TC(F(\RP^n)) \geq \zcl_{\ZZ_2}(SO(n))+\zcl_{\ZZ_2}(\RP^n) + 1 \geq \cl_{\ZZ_2}(SO(n))+\cl_{\ZZ_2}(\RP^n)+1$$
for $n \in \{1,3,7\}$. It is well known that $H^*(\RP^n;\ZZ_2) \cong \ZZ_2[\alpha]/(\alpha^{n+1})$ as rings, where $\alpha$ has degree one. This implies that $\cl_{\ZZ_2}(\RP^n)=n$ for all $n \in \NN$. Thus, we have shown that
$$\cl_{\ZZ_2}(SO(n)) +n+1 \leq \TC(F(\RP^n)) \leq \cat(SO(n))+n \quad \text{for } n \in \{1,3,7\}.$$
Since it was shown in \cite{IwaseMimuraNishi} that $\cat (SO(n))=\cl_{\ZZ_2}(SO(n))+1$ for $n \leq 10$, the lower and the upper bound coincide, such that $\TC(F(\RP^n))=\cl_{\ZZ_2}(SO(n))+n+1$ for $n \in \{1,3,7\}$. Using Theorem \ref{TheoremKorbas}, we explicitly compute that
$$\TC(F(\RP^1)) = 2, \quad \TC(F(\RP^3))=7, \quad \TC(F(\RP^7)) = 19.$$





 \bibliography{TCframe}
 \bibliographystyle{amsalpha}
\end{document}